\documentclass[12pt]{article}
\usepackage{amssymb,amsthm}
\usepackage{graphicx}
\usepackage{amsmath}
\usepackage{mathrsfs}
\usepackage{bm}
\usepackage[show]{ed}
\usepackage{mathtools}
\usepackage{hyperref}
\hypersetup{
    colorlinks=true,       
}

\newcommand{\jump}[1]{[\![#1]\!]}
\newcommand{\ave}[1]{\{\!\{#1\}\!\}}
\newcommand{\triple}[1]{|\!|\!|#1|\!|\!|}

\newtheorem{Theorem}{Theorem}[section]
\newtheorem{Proposition}{Proposition}[section]

\numberwithin{equation}{section}

\title{Time-Dependent Wave-Structure Interaction Revisited: Thermo-piezoelectric Scatterers }\author{{\sc George  C. Hsiao}\thanks{Department of Mathematical Sciences, University of Delaware,
Newark, DE 19716-2553, USA \quad
Email: {\tt ghsiao@udel.edu}} \quad
and \; {\sc Tonatiuh  S\'{a}nchez-Vizuet} \thanks{ Department of Mathematics, The University of Arizona, Tucson, AZ 85721-0089 USA 
Email: {\tt tonatiuh@math.arizona.edu}.}}

\date{\em \small Dedicated to Professor Wolfgang L. Wendland \\ on the occasion of his 85th Birthday}

\begin{document}
\maketitle

\abstract{In this paper, we are concerned with a time-dependent transmission problem for a thermo-piezoelectric elastic body immersed in a compressible fluid. It is shown that the problem can be treated by the boundary-field equation method, provided an appropriate scaling factor is employed. As usual, based on estimates for solutions in the Laplace-transformed domain, we may obtain properties of corresponding solutions in the time-domain without having to perform the inversion of the Laplace-domain solutions. \\}


\section{Introduction}\label{sec:Introduction}
%
The mathematical description of the interaction between an acoustic wave and an elastic body is of central importance in applied mathematics and engineering, as attested for instance by its useage for detection and identification of submerged objects. The problem is mathematically formulated as a transmission problem between elastic and acoustic fields communicating through an interface and is referred to in the literature either as ``fluid-structure interaction problem" or``wave-structure interation problem". The former terminology (wave-structure interaction) is also used to describe a similar problem that involves the coupling between fluid equations (either Stokes or Navier-Stokes) and the equations of elasticity. Here we will be interested in the coupling between the acoustic and elastic wave equations, and we will use the term "wave-structure interaction" exclusively to avoid any confusion.

In the early days of the field, most of the mathematical formulations of these kinds of problems were based on time-harmonic formulations. Motivated by the paper of Mamdi and Jean \cite{HJ:1985},  Hsiao, Kleinman and Schuetz 's paper from 1988 \cite{HKS:1988} gave the first mathematical justification of a variational formulation for wave-structure interaction problems. This set out the field for many further efforts that expanded the understanding of time-harmonic scattering (see, e.g. \cite{BiMa:1991, AmHr:1988, LuMa:1995, ScBe:1989, Hs:1994}). Over the years, time-harmonic wave-structure interaction problems have been studied in various different areas such as inverse problems \cite{EHR:2007, EHR:2008}, interaction of fluid and thin structures \cite{HsNi:2003}, interaction of electromagnetic fields and elastic bodies \cite{CaHs:2002, GHM:2010}, to name just a few. 
  
One of the main reasons behind the use of the boundary-field equation method for treating time-harmonic wave-structure problems is to reduce the transmission problem, posed originally in an unbounded domain, to one set in the bounded domain $\Omega$ determined by the elastic scatterer (see Figure \ref{fig:fig1}). However, the conversion from an unbounded to a bounded domain comes at the price of turning the problem into a non-local one, which brings along some mathematical disadvantages. Since the sesquilinear form arising from the the nonlocal boundary problem can satisfy  only a G\aa rding inequality, in oder to apply the standard Fredholm alternative for the existence theory, the uniqueness of the solution becomes a requirement. However, the straightforward boundary-field method can not circumvent the drawbacks, because the problem is not uniquely solvable when the frequency of the incident wave coincides with what is known as a "Jones frequency". At such a frequencies, the corresponding homogeneous problem may have traction free solutions (a recent discussion on this can be found in \cite{DoNiSu:2019}). Moreover, uniqueness of the solutions to the boundary integral equations may not be guaranteed when the exterior wavenumber coincides with an eigenvalue of the corresponding interior Dirichlet problem (see \cite{HKR:2000}). The issue of non-uniqueness has motivated lots of research, and attempts to overcome these difficulties  have been made with the help of methods such as Schenck's Combined Helmholtz Integral Equation Formulation \cite{Sch:1968} (known commonly as the CHIEF method) and the celebrated formulation by Burton and Miller \cite{BuMi:1971}.

\begin{figure}[h]\centering
\includegraphics[width=0.3\linewidth]{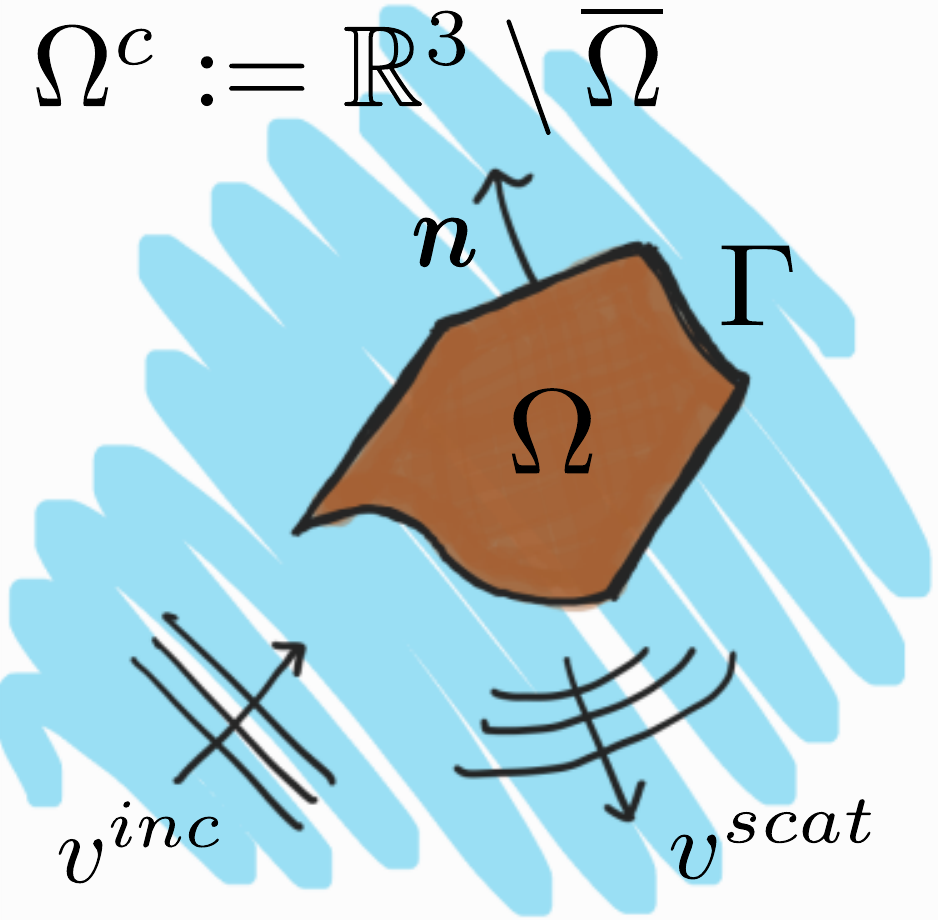} 
\caption{Schematic of the wave scattering problem. The interface between the solid and the fluid is denoted by $\Gamma$, while the outward-pointing normal vector (defined for almost every point in the boundary) is denoted by $\boldsymbol n$.}\label{fig:fig1}
\end{figure} 
 
In the present paper,  inspired by the work of Estorff and Antes \cite{EsAn:1991}, we will apply the boundary-field equation method not to a time-harmonic problem, but rather one in the transient regime. This will require the treatment of the wave equation, as opposed to the Helmholtz equation that is used in the frequency domain. The problem of interest is that of the interaction between a thermo-piezoelectric elastic body immersed in a compressible fluid. The method will not be directly applied in the time-domain, but rather in the Laplace transformed domain. The reasons for this will be made clear in due time. The equations will then be reduced to those of a nonlocal boundary problem in the transformed domain, where all the analysis will be performed. The technique that will be applied will allow us to understand the behavior of the transient problem (and even simulate it computationally if we were so inclined) \textit{without ever having to invert the Laplace transform}.
  
The outline of the solution/analysis procedure for the time-dependent wave-structure interaction is as follows: 

\begin{compactenum}
\item Formulate a time-dependent transmission problem. 
\item Apply the Laplace transform to the time-dependent transmission problem.
\item Reduce the transformed transmission problem to a nonlocal boundary problem in the bounded domain $\Omega$ with the help of a Boundary Integral Equation (BIE). This leads to the boundary-field equation formulation of the problem in the transformed  domain.
\item Obtain estimates of variational solutions of the nonlocal boundary problem in terms of the Laplace transformed variable $s$. 
\item Deduce estimates for the solutions in the time domain from those of the corresponding solutions in the Laplace domain using Lubich's and Sayas's approach for treating BIEs of the convolution type \cite{Lu:1994, LuSc:1992, LaSa:2009b}).
\end{compactenum}

The process described above has been successfully applied to a number of special cases \cite{HaQiSaSa:2015, HsSaSa:2016, HsSaSaWe:2016, SaSa:2016}. However, in all the cases under considerations the formulations in the fluid domain were given in terms of velocity potentials, not in terms of standard fluid pressures. As will be seen, to formulate the problem in term of fluid pressure, an appropriate scaling factor will have to be introduced.

The analysis will proceed more or less following the steps outlined above. The time-domain formulation of the problem is introduced in Section \ref{sec:Formulation}, the corresponding nonlocal boundary problem in the Laplace transferred domain is then described in Section \ref{sec:NonLocalProblem}. Section \ref{sec:VariationalSolutions} contains mathematical ingredients concerning crucial estimates for the solution of the nonlocal boundary problem in the transformed domain. The main results in the time domain are presented in Section \ref{sec:TimeDomain} and end the paper with some concluding remarks on Section \ref{sec:Conclusions}.  
%
\section{Formulations of the problem}\label{sec:Formulation} 
%
We will denote by $\Omega$ an open and bounded subset of $\mathbb{R}^3$ that will be considered to be occupied by an elastic solid. We will further assume that the boundary of the solid is described by a Lipschitz-continuous curve and will denote it by $\Gamma$. The exterior of this solid, which will be denoted by $\Omega^c = \mathbb{R}^3  \setminus  \overline\Omega$, will be filled by an inviscid and compressible fluid. A schematic of the geometric setting is depicted in Figure \ref{fig:fig1}.

We will consider that, when at rest, the velocity, pressure and density in the fluid are described by the constant fields $ {\bf v}_0 = 0, p_0$ and $ \rho_f $, and will be interested in the time evolution of small parturbations from this static configuration as described by the fields ${\bf v}, p $ and  $\rho$ which is given by the linearized Euler equation in the fluid domain $\Omega^c$
 
\begin{equation} \label{eq:2.1} 
\rho_f \frac{\partial {\bf v} } {\partial  t} +  \nabla p =  {\bf 0}, 
\end{equation}

the continuity equation 

\begin{equation}\label{eq:2.2}
\frac{\partial \rho}{\partial t} + \rho_f \, \nabla\cdot{\bf v} = 0,
\end{equation}

for $\rho$, and ${\bf v}$, and the state equation for $p$ and $\rho$

\begin{equation}\label{eq:2.3}  
p = c^2 \rho.
\end{equation}

Above, the sound speed $c$ is a function that varies depending on the properties of the fluid (see e.g., \cite{Ac:1990, Se:1958}), and the operator $\frac{\partial}{\partial t}$ is the usual partial derivative with respect to the time variable, not to be confused with the \textit{material derivative}. All these equations are posed in $\Omega^c \times [0, \infty)$. A simple manipulation shows  that with the help of \eqref{eq:2.2} and \eqref{eq:2.3}, we may replace equation \eqref{eq:2.1} by a single wave equation for the pressure $p$

\[ 
\frac{1}{c^2} \frac{ \partial ^2 p}{\partial t ^2}  - \Delta p = 0  \quad \mbox{in} \quad \Omega^c  \times [0, \infty) .
\] 
 
Now, inside the domain $\Omega$ occupied by the solid, the governing equation depends on the properties of the solid. It may be as simple as an elastic obstacle, or it may have more complicated physical properties such as a thermoelastic solid, or a thermopiezoelectric solid as in our present case. The problem under consideration, for a thermo-piezoelectric body, consists of determining the stress and strains tensors,  $\bm{\sigma}(\boldsymbol x, t)$ and $\bm {\varepsilon}(\boldsymbol x, t)$, the elastic displacement $ {\bf  u} (\boldsymbol x, t) $, temperature variation $\theta (\boldsymbol x,t) $ and the electric potential $\varphi(\boldsymbol x, t)$. The physics of the process can be described in terms of the reference density of the solid $\rho_e$, the absolute temperature in the solid $T$ and its stress-free reference temperature $T_0$, the electric displacement vector $\mathbf D(\boldsymbol x,t)$, and the entropy per unit volume $P(\boldsymbol x,t)$. The governing equations have been derived by Mindlin \cite{Mi:1961} and consist of three coupled partial differential equations, namely the dynamic elastic equations

\begin{equation}\label{eq:elastic}
\rho_{e} \;  \frac{\partial^2 \bf u} {\partial t^2} -  \nabla\cdot  \bm{\sigma}  = {\bf 0}, 
\end{equation}

the generalized heat equation

\begin{equation}\label{eq:thermal}
T\;  \frac{\partial P}{\partial t} - \Delta\, \theta  =  0,  \quad \theta := T - T_0,
\end{equation}

and the equation of the quasi-stationary electric field (i.e., Gauss's electric field law without electric charge density):

\begin{equation}\label{eq:electric}
 \nabla\cdot{\bf D}  =  0. 
\end{equation}

These equations need to be supplied with adequate constitutive relations providing a description of the functional dependence between the unknown variables within the thermo-piezoelectric media. In the isotropic case, the constitutive relations may be simplified in the form (see \cite{Ku:1979}): 

\begin{align}
\label{eq:2.8a} 
\bm{\sigma}& =  \bm{\sigma}({\bf u}, \theta, \varphi) :=\bm{\sigma}_{e}({\bf u} ) - \zeta \, \theta\, {\bf I} - {\bf e}^{\top}{\bf E}, \\
\nonumber
P & = P({\bf u}, \theta, \varphi) :=  \zeta\,  \nabla\cdot{\bf u} +  \frac{c_\epsilon} {T_0} \theta
 + {\bf p} \cdot {\bf E},\\
\label{eq:2.10a}   
{\bf D} & = {\bf D}( {\bf u}, \theta, \varphi ) : = \; {\bf e} \,  \bm {\varepsilon} ({\bf u}) \,  +  \theta \, {\bf p} 
+ \epsilon \, {\bf E},
\end{align}

where 
 
\[
\bm{\sigma}_{e}:= \lambda \, (\nabla\cdot{\bf u} ) \, {\bf I} + 2\,  \mu \, {\bm \varepsilon }( {\bf u}), \quad  \mbox{and} \quad \bm{\varepsilon}( {\bf u}) : = \frac{1}{2} ( \nabla{\bf u} + \nabla {\bf u}^{\top} )
\]

are the usual stress and strain tensors for isotropic elastic media, while ${\bf e} = (( \mathrm e_{ijk})) $ is the piezoelectric tensor with constant elements such that $\mathrm e_{kij} = \mathrm e_{kji}$. This third order tensor maps matrices into vectors, while its adjoint, which will be denoted by ${\bf e}^{\top}$, maps vectors into symmetric matrices. More precisely, for a real symmetric matrix $\mathbf M\in \mathbb R^{d\times d}_{sym}$ and for a vector, $\mathbf d \in \mathbb R^d$ we define

\[
(\mathbf e\mathbf M)_{k} := \sum_{ij}\mathrm e_{kij}\mathbf M_{ij} \in \mathbb R^d
\quad \text{ and } \quad
(\mathbf e^\top\mathbf d)_{ij} := \sum_{k}\mathrm e_{kij}\mathbf d_k \in \mathbb R^{d\times d}_{sym}.
\]

The constants $\zeta$ and $\epsilon $ are respectively  the thermal and dielectric constants;  $ c_{\epsilon}$ is the specific heat at constant strain, and the constant vector ${\bf p} $ is the pyroelectric moduli vector. The electric field  {\bf E} in the constitutive equations is replaced by ${\bf E} = - \nabla \varphi$. As usual, $\mu>0$ and $\lambda$ are the Lam\'{e} constants for the elastic body (note that it is customary to require $\lambda>0$, however this is not necessary as long as the physically meaningful quantity $3 \lambda +2 \mu$, known as the bulk modulus, remains positive). The theory of thermopiezoelectricity was first proposed by Mindlin \cite{Mi:1961}. The physical laws for thermopiezoelectric materials were explored by Nowacki \cite{No:1978} (GCH would like to thank Prof. T.W. Chou for locating this reference for him),\cite{No:1978a}, where more general constitutive relations are available than those given in \eqref{eq:2.8a} - \eqref{eq:2.10a}. 
 
Making use of these constitutive relations in conjunction with the governing equations \eqref{eq:elastic}, \eqref{eq:thermal}, and \eqref{eq:electric} we arrive at  differential equations
   
\begin{align}
\rho_{e} \;  \frac{\partial^2 \bf u} {\partial t^2} - \nabla\cdot\left(\bm{\sigma}_{e} \, ( {\bf u } ) -  \zeta \theta\,  { \bf I }  + { \bf e}^{\top} \nabla \varphi \right)  =\,& {\bf 0} \label{eq:2.11a}\\ 
\frac{\partial }{\partial t} \left(\zeta \, \nabla\cdot{\bf u}  - {\bf p } \cdot \nabla \varphi \right)   + \frac{1}T_0 \left(   c_{\epsilon} \frac{\partial \theta} {\partial t}    - \Delta \theta \right)   =\,& 0 \label{eq:2.12a} \\
   \nabla\cdot\left( {\bf e } \, \bm{\varepsilon} ( {\bf u} ) +  \theta {\bf p} 
 - \epsilon\,  \nabla \, \varphi  \right)   =\,& 0 \label{eq:2.13a}
\end{align}
 
We remark that Equation \eqref{eq:2.12a} is derived under the assumption that   
$ | \frac{\theta} {T_0}| \ll 1 $. This means $ T \simeq T_0$, since $ T = T_0 ( 1 + \frac{\theta}{T_0} )$. Equations \eqref{eq:2.11a} -\eqref{eq:2.13a} constitute the complete set of equations of thermopiezoelectricity coupling a hyperbolic equation for ${\bf u}$, a parabolic equation for $\theta$, and an elliptic equation for $\varphi$. Here and in the sequel, all the constant physical quantities satisfy

\[ 
\rho_{e} > 0,\,  \mu > 0, \,  3 \lambda +2 \mu > 0, \,  \mathrm e_{ijk} > 0,\,  \zeta > 0, \,  c_{\varepsilon} > 0.
 \]

To formulate a typical time-dependent wave-structure problem, we need to prescribe initial, boundary and transmission conditions. This leads to a model of partial differential equations for the time-dependent wave-structure problem.\\[2mm]

\noindent
{\bf Time-dependent transmission problem.} {\it Given $ ( p^{inc}, \partial_n p^{inc}, f_{\theta}, f_{\bf D} ) $, find the solutions $ ({\bf u}, \theta, \varphi) $ in $\Omega \times [0, \infty)$,  and $p$ in $ \Omega^c \times [0, \infty) $  satisfying the partial differential equations}

\begin{alignat}{6}
\label{eq:2.14a}
\rho_{e} \;  \frac{\partial^2 \bf u} {\partial t^2} - \nabla\cdot\left(\bm{\sigma}_{e} \, ( {\bf u } ) -  (\zeta \theta) { \bf I }  + { \bf e}^{\top} \nabla \varphi \right)  =\,& {\bf 0} && \quad \mbox{ \textit{in}} \quad \Omega \times [0, \infty) \\ 
\nonumber
\frac{\partial }{\partial t} \left(\zeta \, \nabla\cdot{\bf u}  - {\bf p } \cdot \nabla \varphi \right)   + \frac{1}{T_0 } \left(   c_{\epsilon} \frac{\partial \theta} {\partial t}    - \Delta \theta \right)   =\,& 0  && \quad  \mbox{ \textit{in}} \quad \Omega \times [0, \infty)  \\
\nonumber
    \nabla\cdot\left( {\bf e } \, \bm{\varepsilon} ( {\bf u} ) +  \theta {\bf p} 
 - \epsilon\,  \nabla \, \varphi  \right)   =\,& 0  && \quad  \mbox{ \textit{in}} \quad \Omega \times [0, \infty)  
 \end{alignat}

{\it and}
 
\begin{equation}
\label{eq:2.15}
\frac{1}{c^2} \frac{ \partial ^2 p}{\partial  t ^2}  - \Delta p = 0  \quad \mbox{\textit{in}} \quad \Omega^c  \times [0, \infty) .
\end{equation}

{\it together with the transmission conditions} 
 
\begin{alignat}{6}
\label{eq:2.16}
\bm {\sigma}({\bf u}, \theta, \varphi) ^- {\bf n}   =\,& - \left( p + p^{inc}\right)^+ {\bf n} && \quad  \mbox{on} \quad \Gamma \times [0, \infty),  \\
\nonumber
\frac{\partial  {\bf u}^-  } {\partial t} \cdot {\bf n}   =\,& - \frac{1} {\rho_f}  \int_0^t   \frac{\partial}{\partial n} \left( p + p^{inc} \right)^+  d\tau  && \quad  \mbox{on} \quad \Gamma \times [0, \infty), 
\end{alignat}  
 
{\it the  boundary conditions } 

\begin{equation}
 \partial_n \theta = f_{\theta}, \quad {\it and } \quad {\bf D} \cdot n = f_{\bf D} \quad \mbox{on} \quad \Gamma \times 
 [0, \infty) \label{eq:2.17} 
\end{equation}

{\it and homogeneous initial conditions for $ {\bf u}, \; \partial {\bf u}/ \partial t, \theta, p $ 
and $\partial p / \partial t $}.

The given data and solutions are required to satisfy certain regularity properties that will be specified later. In the formulation, we use the superscripts $^+$ or $^-$ to denote the \textit{traces} or restrictions to the boundary $\Gamma$ of a function when taken as limits from functions defined on $\Omega^c$ and $\Omega$, respectively. This is equivalent to the notation $v^+ = \gamma^+ v$ and $v^-= \gamma^- v$ customary in the mathematical literature.  Whenever the trace---or restriction---of a function to the boundary does not depend on the side from which the limit is taken, we will drop the superscript and write only $\gamma v$. In this  formulation,  one  has to solve the wave equation for the pressure in the exterior---unbounded---domain, which can be a drawback from the  computational point of view. 

To sidestep the challenge of undboundedness, we will resort to a formulation of the transmission problem defined by \eqref{eq:2.14a}, \eqref{eq:2.15}, \eqref{eq:2.16}, and \eqref{eq:2.17} that will couple boundary integral equations with partial differential equations. This technique, put forward in the context of time-harmonic problems \cite{HKR:2000}, transforms the problem into a nonlocal one that is posed only in the bounded computational domain $\Omega$ by representing the pressure in the fluid domain through an integral along the interface $\Gamma$ between the solid and the fluid. To this avail, we must introduce the fundamental solution to the wave equation

 \[ 
\mathcal G (\boldsymbol x-\boldsymbol y, t) = \frac{1}{4\pi |\boldsymbol x - \boldsymbol y|} \,  \delta( t - c^{-1} | \boldsymbol x - \boldsymbol y|).
 \]
 
Above, $\delta(\cdot)$ is Dirac's delta. Using this fundamental solution, it is possible to express any solution to \eqref{eq:2.15} in terms of density functions $\phi$, and $\lambda$ that correspond to the Cauchy data of the problem. Namely the pressure restricted to $\Gamma$ and its normal derivative respectively. This is known as the Kirchhoff representation formula (see e.g., \cite{HsWe:2011a,HsWe:2011b,LaSa:2009b})

 \begin{equation}  
 p(\boldsymbol x,t) =  (\mathcal D * \phi)(\boldsymbol x,t) - ( \mathcal{S} * \lambda)(\boldsymbol x, t) , \quad  (\boldsymbol x,t) \in \Omega^c \times [0, \infty). \label{eq:2.18}
\end{equation}

Above, the asterisk $*$ refers to convolution with respect time time,

\[
f*g = \int_0^t f(t - \tau ) g(\tau) d\tau,
\]

and $\mathcal{D}$ and $\mathcal{S}$ are known respectively as the double- and simple-layer potentials. They can be defined as convolutions with the fundamental solution and its normal derivative

\begin{align*}
( \mathcal{S} * \lambda)(\boldsymbol x, t)  :=\,& \int_0^t \, \int_\Gamma \mathcal G (\boldsymbol x -\boldsymbol y, t- \tau)  \lambda( \boldsymbol y, \tau )\, d \Gamma_{\boldsymbol y}
d\tau \\
=\,& \int_{\Gamma} \frac{1}{4\pi| \boldsymbol x-\boldsymbol y|}\lambda( \boldsymbol y, t - c^{-1} |\boldsymbol x - \boldsymbol y|) \, d\Gamma_{\boldsymbol y} \\
=\,& \int_{\Gamma} E(\boldsymbol x, \boldsymbol y) \lambda( \boldsymbol y, t - c^{-1} |\boldsymbol x - 
\boldsymbol y|) \, d\Gamma_{\boldsymbol y} , \\
 (\mathcal D * \phi)(\boldsymbol x,t) :=\,& \int_0^t \int_\Gamma \frac{\partial}{\partial n_{\boldsymbol y}}  \mathcal G (\boldsymbol x - \boldsymbol y, t- \tau) \phi(\boldsymbol y, \tau) d \Gamma_{\boldsymbol y} d \tau \\
=\,& \int_\Gamma \frac{\partial}{\partial n_{\boldsymbol y}}  \left(\frac{1}{4\pi|\boldsymbol x-\boldsymbol y|} \phi(\boldsymbol y, t - c^{-1} |\boldsymbol x - \boldsymbol y| )\right) d \Gamma_{\boldsymbol y}  \\
=\,&  \int_\Gamma \frac{\partial}{\partial n_{\boldsymbol y}} \left( E (\boldsymbol x , \boldsymbol y )\, \phi(\boldsymbol y, t - c^{-1} |\boldsymbol x - \boldsymbol y| )\right)\,  d \Gamma_{\boldsymbol y}.
 \end{align*}
  
In these equations, we have denoted the fundamental solution of the negative Laplacian in $\mathbb{R}^3$ by $E(\boldsymbol x, \boldsymbol y):= \frac{1}{4\pi|\boldsymbol x-\boldsymbol y|}$. The reader will notice that the convolution with the fundamental solution introduces a delay into the density functions $\lambda$ and $\phi$. It is customary in the wave propagation community, to write  $[\varphi] = \varphi( \boldsymbol y, t -c^{-1} |\boldsymbol x - \boldsymbol y|)$ and call $[\varphi] $ the retarded value  of $\varphi$. This is the reason why sometimes $( \mathcal{S} * \lambda)(\boldsymbol x, t) $ and $ (\mathcal D * \phi)(\boldsymbol x,t)$ are referred to as the \textit{retarded layer potentials}.

Similarly, by introducing  the convolution integral
 
\[
( \mathcal{I}* \varphi) (\boldsymbol x, t)  :=  \int_0^t \int_{\Gamma} \delta (\boldsymbol x- \boldsymbol y; t-\tau) \varphi(\boldsymbol y, \tau) d\Gamma_{\boldsymbol y} d\tau = \varphi(\boldsymbol x, t),
\]

At non-singluar points of $\Gamma$, the Cauchy data $\phi$ and $\lambda$  satisfy the following system of boundary integral equations (see, e.g., \cite{ BaHa:1986a, BaHa:1986b, Co:2004,HsWe:2008})
 \begin{equation}\label{eq:2.20}
\left(
\begin{matrix}
\phi\\[3mm]
\lambda
\end{matrix}
\right)
= \left (
\begin{matrix}  
      \frac{1}{2}\mathcal{I}  + \mathcal{K } &-\mathcal{V} \\[3mm]
       -\mathcal{W}   &   \left(\frac{1}{2} \mathcal{I} - \mathcal{K}\right)'  \\
    \end{matrix}
    \right ) * \left(\begin{matrix}
\phi \\[3mm]
\lambda 
\end{matrix}
\right)
\quad \mbox{on} \quad \Gamma \times [0, \infty). 
 \end{equation}

The boundary integral operators $\mathcal{V}, \mathcal{K}, \mathcal{K}^\prime,$ and $ \mathcal{W}$ appearing above are known respectively as the simple layer, double layer, transpose double layer, and hypersingular boundary integral operators for the dynamic wave equation. They are defined as follows

\begin{equation*}
\left.
\begin{aligned} 
(\mathcal {V}  * \lambda )  :=\,& \ave{\gamma ( \mathcal{S}* \lambda)} 
=   \frac{1}{2} \left(  \gamma^- (\mathcal{S} * \lambda )  +  \gamma^+ (\mathcal{S} * \lambda ) \right)\\
=\, & \gamma^- (\mathcal{S} * \lambda ) = \gamma^+ (\mathcal{S} * \lambda)  \\
(\mathcal{K}*\phi) :=\, & \ave{ \gamma ( \mathcal{K}* \phi) } = \frac{1}{2} \left(  \gamma^- (\mathcal{K} * \phi )  +  \gamma^+ (\mathcal{K} * \phi) \right)\\
(\mathcal{K^{\prime} }* \lambda) :=\, & \ave{ \gamma ( \mathcal{K^{\prime}}* \lambda } = \frac{1}{2} (  \gamma^- (\mathcal{K^{\prime}} * \lambda )  +  \gamma^+ (\mathcal{K^{\prime}} * \lambda ) )\\
(\mathcal{W}* \phi) :=\, &  - \ave{\partial_n( \mathcal{D} * \phi)} 
= - \frac{1}{2} \left(  \partial_n^- (\mathcal{D} * \lambda )  +  \partial_n^+ (\mathcal{D} * \lambda ) \right)\\
= \, & - \partial_n^- (\mathcal{D} * \phi) = - \partial_n^+ (\mathcal{D} * \lambda)  \\ 
\end{aligned}  
\right\}   \quad \mbox{on} \quad  \Gamma\times [0, \infty).
\end{equation*} 

Note that the averaging operator $\ave{\cdot}$ has been defined implictly in the second equality on the first line above. 

We can now state the reformulation of the original problem that we will be focusing on:

\noindent
 {\bf  Time-dependent nonlocal problem}. {\it Given $ ( p^{inc}, \partial_n p^{inc}, f_{\theta}, f_{\bf D} ) $, find the solutions $ ({\bf u}, \theta, \varphi) $ in $\Omega \times [0, \infty) $  and $(\phi, \lambda) $ on $ \Gamma \times [0, \infty) $  satisfying the partial differential equations}
 
\begin{alignat*}{6}
\rho_{e} \;  \frac{\partial^2 \bf u} {\partial t^2} - \nabla\cdot\left( \bm{\sigma}_{e} \, ( {\bf u } ) -  (\gamma \theta) { \bf I }  + { \bf e}^{\top}\nabla \varphi \right)  = \,&{\bf 0} \quad \mbox{\textit{in}} \quad \Omega \times [0, \infty),\\ 
\frac{\partial }{\partial t} \left( \gamma \, \nabla\cdot{\bf u}  - {\bf p } \cdot \nabla \varphi \right)  + \frac{1}{T}_0  \left(   c_{\epsilon} \frac{\partial \theta} {\partial t}    - \Delta \theta \right)   =\,& 0 \quad \mbox{\textit{in}} \quad \Omega \times [0, \infty), \\
   \nabla\cdot\left( {\bf e } \, \bm{\varepsilon} ( {\bf u} ) +  \theta {\bf p} 
 - \epsilon\,  \nabla \, \varphi  \right)   =\,& 0 \quad \mbox{\textit{in}} \quad \Omega \times [0, \infty),
\end{alignat*}

\textit{and the differential- boundary integral equations}

\begin{alignat*}{6}
- \rho_f \frac{\partial {\bf u}} {\partial t}\! \cdot\! {\bf n}\! + \! \! \int_0^t \!\! \left(\! \! (\mathcal{W} \!*\! \phi)(\boldsymbol x, t)\! -\! \frac{1}{2} \lambda(\boldsymbol x,t)\! + \!( \mathcal{K}^{\prime}\! *\! \lambda) (\boldsymbol x, t)\!\! \right) d\tau \!= \,& \!\! \int_0^t \!\! \partial^+_n p^{inc} d\tau  &&\;\;\; \mbox{\textit{on}}  \; \Gamma\! \times\! [0, \infty),\\
 \frac{1}{2}\phi(\boldsymbol x,t)\! -\! (\mathcal{K} \!*\!\phi) (\boldsymbol x, t) \! + \! (\mathcal{V} \!*\! \lambda )(\boldsymbol x, t) \! =\,& \!  0 &&\;\;\; \mbox{\textit{on}}  \; \Gamma \!\times\! [0, \infty).
 \end{alignat*}

\textit{together with the transmission condition} 

\[
 \bm {\sigma}({\bf u}, \theta, \varphi) ^- {\bf n}   = - \left( \phi + p^{inc}\right)^+ {\bf n}, \quad \mbox{\textit{on}}  \; \Gamma \times [0, \infty),
 \]
 
 \textit{the boundary conditions} 
 
\[
 \partial_n \theta = f_{\theta}, \quad {\it and } \quad {\bf D} \cdot n = f_{\bf D} \quad \mbox{\textit{on}} \; \Gamma \times 
 [0, \infty),
 \]
 
{\it as well as homogeneous initial conditions for $ {\bf u}, \; \partial {\bf u}/ \partial t, \; \theta, \phi $ 
and $\lambda $.} 

Throughout the paper, the given data ($p^{inc}, \partial_n p^{inc}, f_{\theta}, f_D $) will always be assumed to be \textit{causal functions}. Namely, functions of time $t$ that vanish identically for $t  < 0$. 

From the definitions of  the operators $\mathcal{V}, \mathcal{K}, \mathcal{K}^\prime,$ 
and $ \mathcal{W}$, we notice that the non-locality of the boundary integral equations in \eqref{eq:2.20} is not restricted to space, but extends also into the time variable. 

To study the well-posedness of this formulation, we will first transform it to the Laplace domain, where the analysis will be performed. This idea is due to Lubich and Schneider (see, e.g. \cite{Lu:1994, LuSc:1992}) and has been extended by Laliena and Sayas \cite{LaSa:2009b, Sayas:2016}. We remark that, the passage to the Laplace domain is required only to simplify the analysis and the stability estimates, but for a computational implementation this technique \textit{does not} require the numerical inversion of the Laplace transform. Instead, from the estimates of the solutions in the transformed domain, the properties of the solutions in the time domain will be deduced automatically. The later is particularly desirable from the computational point of view. In the next section, we will consider the the model of partial differential equations for the time-dependent wave-structure problem and/or the time-dependent nonlocal boundary transmission problem in the Laplace domain.

\section{A nonlocal boundary problem}\label{sec:NonLocalProblem}
%
The passage to the Laplace domain will require us to  first introduce some definitions. The complex plane be denoted in the sequel by $\mathbb{C}$, while we will use the notation 

\[
\mathbb{C}_+:= \{ s \in \mathbb{C} : \mathrm{Re} s > 0\},
\]

to refer to the positive half plane. For any complex-valued function with limited growth at infinity $f \colon [0, \infty) \to \mathbb{C}$, its Laplace transform is given by 

\[
\widehat{f}(s)= \mathcal{L}{f}(s) := \int_0^\infty e^{-st} f(t) dt, 
\]

whenever the integral converges. A broad class of functions for which the Laplace transform is well-defined is that of functions of exponential  order. More precisely, a function $f$ is said to be of exponential order if there exist constants $t_0 >0$, $ M \equiv M(t_0)>0$,  and  $\alpha \equiv \alpha(t_0)>0$ satisfying

\[
t \geq t_0  \Longrightarrow |f(t)| \leq M e^{\alpha t}.
\]

In the following, let $\widehat{\mathbf{u}}(s) :=  \mathcal{L}\{ {\bf u}(\boldsymbol x,t)\}, \widehat{ \theta}(s):= \mathcal{L}\{{\theta(\boldsymbol x,t)} \}, \widehat{\varphi}(s) :=  \mathcal{L}\{{\varphi(\boldsymbol x,t)} \}$, and   $\widehat{p}(s):= \mathcal{L}\{ p(\boldsymbol x, t) \}$.  Then, in the Laplace domain, equations \eqref{eq:2.14a}, \eqref{eq:2.15}, and \eqref{eq:2.16}  become 

\begin{alignat}{6} 
\label{eq:3.1a}
 - \nabla\cdot\left( \bm{\sigma}_{e} \, ( \widehat{ {\bf u }} ) -  (\zeta \widehat{\theta) }{ \bf I }  + { \bf e}^{\top} \nabla\widehat{ \varphi} \right) +  \rho_{e}s^2 \widehat{{\bf u}}  =\,& {\bf 0} \quad \mbox{in} \quad \Omega \\
 \label{eq:3.1b} 
s\left( \zeta \, \nabla\cdot\widehat{{\bf u}}  - {\bf p } \cdot \nabla \widehat{\varphi }\right)   + \frac{1}{T_0}  \left(   - \Delta \widehat{\theta}  + c_{\epsilon} s \,  \widehat{\theta}\right)   = \,&0 \quad \mbox{in} \quad \Omega \\
\label{eq:3.1c}
   \nabla\cdot\left( {\bf e } \, \bm{\varepsilon} ( \widehat{ {\bf u} }) + \widehat{ \theta} {\bf p} 
 - \epsilon\,  \nabla \, \widehat{\varphi } \right)   =\,& 0 \quad \mbox{in} \quad \Omega
\end{alignat}
 
and 

\begin{equation} 
 - \Delta \widehat{p}  + \frac{s^2}{c^2}\,  \widehat{p} = 0  \quad \mbox{in} \quad \Omega^c . \label{eq:3.2}
\end{equation}

together with the transmission conditions 

\begin{alignat}{6}
\label{eq:3.3a}
 \bm {\sigma}(\widehat{{\bf u}}, \widehat{\theta}, \widehat{\varphi} ) ^- {\bf n}   = - \left( \widehat {p} +\widehat{ p}^{inc}\right)^+ {\bf n} \quad \mbox{on} \quad \Gamma, \\
 \label{eq:3.3b}
s^2 \,  \widehat{{\bf u}} \cdot {\bf n}   = - \frac{1} {\rho_f}  \frac{\partial}{\partial n} \left(\widehat{ p}  +  \widehat{p} ^{inc} \right)^ {+} \quad \mbox{on} \quad \Gamma,
\end{alignat}  

and the  boundary conditions 

\begin{equation}
 \partial_n \widehat{\theta} = \widehat{f}_{\theta}, \quad \mbox{ and } \quad \widehat{{\bf D}} \cdot n = \widehat{f}_{\bf D} \quad \mbox{on} \quad \Gamma. \label{eq:3.4}
 \end{equation}
 
Above, analogously to the time-domain system, the generalized stress tensor is given by $\bm {\sigma} ( \widehat{{\bf u}}, \widehat{\theta}, \widehat{\varphi} )  := \bm{\sigma}_{e} \, ( \widehat{ {\bf u }} ) -  (\zeta\widehat{\theta}){ \bf I }  + {\bf e}^{\top} \nabla\widehat{ \varphi} $.

We will make use of Green's third identity to derive the equivalent non-local problem. First we must represent the solutions of \eqref{eq:3.2} in the form: 

\begin{equation}\label{eq:3.5}
 \widehat{p}(s)  = D(s) \widehat{\phi} - S(s)\widehat{\lambda} \quad \mbox{in} \quad \Omega^c,
 \end{equation}
 
where the Cauchy data for (\ref{eq:3.2}) is given by the densities $\widehat{\phi}:= \widehat{p}^+(s)$ and $\widehat{\lambda}:= \partial \widehat{p}^+ /\partial n$, and the simple-layer, $S(s)$, and double-layer, $D(s)$, potentials of the corresponding operator defined by 

\begin{eqnarray*}
S(s) \widehat{\lambda} (\boldsymbol x) &:=& \int_\Gamma  E_{s/c}(\boldsymbol x,\boldsymbol y) \widehat{\lambda}(\boldsymbol y) d\Gamma_{\boldsymbol y}, 
 \quad \boldsymbol x \in \Omega^c,\\
 \nonumber
D(s) \widehat{\phi} (\boldsymbol x)  &:=& \int_\Gamma \frac{\partial}{\partial n_y} E_{s/c}(\boldsymbol x,\boldsymbol y) \widehat{\phi}(\boldsymbol y)  d\Gamma_{\boldsymbol y}, \quad \boldsymbol x \in \Omega^c.
\end{eqnarray*}

Here 

\[
E_{s/c}(\boldsymbol x,y) :=  \frac{ e^{- s |\boldsymbol x-\boldsymbol y|/c} }{4 \pi |\boldsymbol x-\boldsymbol y|}
\]

is the fundamental solution of equation (\ref{eq:3.2}). As with their counterpart in the frequency-domain, the Cauchy data $\widehat{\lambda}$ and $\widehat{\phi} $ satisfy the following integral relations:
 
 \begin{equation}\label{eq:3.8}
\begin{pmatrix}
\widehat{\phi} \\[3mm]
\widehat{\lambda}\\
\end{pmatrix}
= \left (
\begin{matrix} 
      \frac{1}{2}I + K(s) & -V(s) \\[3mm]
       -W(s)  &  \left( \frac{1}{2}I - K(s) \right)'  \\
    \end{matrix}
    \right )\begin{pmatrix}
\widehat{\phi} \\[3mm]
\widehat{\lambda}\\
\end{pmatrix}
\quad \mbox{on} \quad \Gamma.
 \end{equation}
 
In the preceding relation, $V, K, K^\prime $ and $W$  are the four  basic boundary integral operators defined by

\[
\left.
\begin{aligned} 
V(s) :=\,& \ave{\gamma S(s) } 
=  \frac{1}{2} \Big(  \gamma^- S(s)   +  \gamma^+ S(s)  \Big)\\
 =\,& \gamma^- S(s) = \gamma^+ S(s)  \\
K(s)  :=\,& \ave{ \gamma D(s) } = \frac{1}{2} \Big(  \gamma^- D(s) )  +  \gamma^+ D(s) \Big)\\
K^{\prime}(s) :=\,& \ave{ \gamma S(s) } = \frac{1}{2} \Big(  \gamma^-  S(s)   +  
\gamma^+ S(s) \Big) \\
W (s) :=\,&  - \ave{ \partial_n D(s) } 
= - \frac{1}{2} \Big(  \partial_n^- D (s)  +  \partial_n^+ D(s)  \Big)\\
=\,& - \partial_n^- D(s)  = - \partial_n^+ D(s)  \\ 
\end{aligned}  
\right\}   \quad \mbox{on} \quad \Gamma.
\] 

In terms of $\widehat{\phi}$ and $\widehat{\lambda}$, the two transmission conditions \eqref{eq:3.3a} and \eqref{eq:3.3b} become
 
\begin{alignat}{6}
\nonumber
 \bm {\sigma} (\widehat{{\bf u}}, \widehat{\theta}, \widehat{\varphi} ) ^- {\bf n}   =\,& - \left( \widehat {\phi} (s) + \widehat{ p}(s) \,  ^{inc}\right)^+{\bf n}  &&\quad \mbox{on}\quad \Gamma, \\
 \label{eq:3.10b}
-s^2 \widehat{\mathbf{u}}^-  \cdot \mathbf{n} + \frac{1}{\rho_f} \left( W(s) \widehat{\phi}    - \left( \frac{1}{2} I -  K(s) \right)'  \widehat{\lambda} \right) =\,&
\frac{1}{\rho_f} \Big(\frac{\partial \widehat{ p}^{ \,inc }}{\partial n}\Big)^+ &&\quad \mbox{on}\quad \Gamma.
\end{alignat}
 
Using the densities $\widehat{\phi}$ and $\widehat{\lambda}$ as new unknowns, equation \eqref{eq:3.2} may be eliminated from the problem by using the second equation above together with the boundary integral equation in the first row of (\ref{eq:3.8}), namely

\begin{equation} \label{eq:3.11}
\left( \frac{1}{2}I - K(s) \right) \widehat{\phi} + V (s)\widehat{\lambda} = 0 \quad \mbox{on} \quad \Gamma.
\end{equation}

This leads to an integro-differential formulation for the unknowns $(\widehat{\bf u}, \widehat{\theta}, \widehat{\varphi}, \widehat{\phi}, \widehat{\lambda})$ satisfying the partial differential equations \eqref{eq:3.1a}, \eqref{eq:3.1b}, and \eqref{eq:3.1c} in $\Omega$ together with the boundary conditions $\eqref{eq:3.3b}$, and \eqref{eq:3.4}, and the boundary integral equations $\eqref{eq:3.10b}$ and \eqref{eq:3.11} on $\Gamma$.  

Let us first define the space

\[
 H^1_*(\Omega) := \left\{ \varphi \in H^1(\Omega) \, | \, \int_{\Omega} \varphi\,  (\boldsymbol x) \, dx = 0 \right\}, 
\] 

and restrict our search for the unknown functions $(\widehat{\bf u}, \widehat{\theta}, \widehat{\varphi} )$ to the product space ${\bf H}^1(\Omega) \times H^1(\Omega) \times H^1_*(\Omega) $. To do so, we multiply equations  $\eqref{eq:3.1a}, \eqref{eq:3.1b}$, and $\eqref{eq:3.1c}$ by $( \widehat{\bf v}, \widehat{\vartheta}, \widehat{\psi} )  \in {\bf H}^1(\Omega) \times H^1(\Omega) \times H^1_*(\Omega) $. Integrating by parts the resuting relations will lead to:

\begin{equation}
\left.
\begin{aligned}
a (\widehat{\bf u}, \widehat{\bf v}; s) - \zeta (\widehat{\theta}, \nabla\cdot\widehat{\bf v}) _{\Omega} + (\nabla \widehat{\varphi},  {\bf e}\,  \bm{\varepsilon}(\widehat{\bf v} ))_{\Omega} + \langle \widehat{\phi}\,  {\bf n}, \widehat{\bf v}^-\rangle_{\Gamma}  =\,& - \langle \widehat{p} ^{\, inc\, +} {\bf n}, \widehat{\bf v} ^-\rangle_{\Gamma} \\
 s(\zeta \nabla\cdot\widehat{\bf u}\,  -  {\bf p}\cdot \nabla\; \widehat{\varphi}, \widehat{\vartheta} \; )_{\Omega}  + \frac{1} {T_0} 
 b \, ( \widehat{\theta}, \widehat{\vartheta}; s)_{\Omega} =\,& \frac{1}{T_0}  \langle  \widehat{f} _{\theta}, \widehat{\vartheta}^- \rangle_{\Gamma} \\
- ( {\bf e}\,  \bm{\varepsilon}( \widehat{\bf u}), \nabla \widehat{\psi})_{\Omega} -  (\widehat{\theta} \, {\bf p}, \nabla \widehat{\psi} )_{\Omega} + \epsilon c (\widehat{\varphi}, \widehat{\psi}; s)_{\Omega} =\,& - \langle \widehat{f} _{\bf D}, \widehat{\psi}^{-}  \rangle_{\Gamma} 
\end{aligned} \right\} \label{eq:3.13}
 \end{equation} 
 
where $a(\cdot, \cdot; s), b(\cdot, \cdot; s)$  and  $ c(\cdot, \cdot;, s)$  are sesquilinear forms defined respectively by
 
\begin{align*} 
a(\widehat{\bf u}, \widehat{\bf v}; s) :=\,& (\bm{\sigma}_{e} (\widehat{\bf u}), \bm{\varepsilon} (\widehat{\bf v})) +s^2 \rho_{e} ( \widehat{\bf u}, \widehat{\bf v})_{\Omega} \\ 
 b \, ( \widehat{\theta}, \widehat{\vartheta}; s)_{\Omega}  :=\,& (\nabla  \widehat{\theta}, \nabla \widehat{\vartheta})_{\Omega} +c_{\varepsilon} s (\widehat{\theta}, \widehat{\vartheta})_{\Omega}\\ 
  c (\widehat{\varphi}, \widehat{\psi}; s)_{\Omega} :=\,& ( \nabla\widehat{\varphi}, \nabla\widehat{\psi}) _{\Omega}.
\end{align*} 
 
Now let  $\boldsymbol{A}_{s}, B_{s}$ and $ C_{s}$ be the operators defined by the mappings

\[
\boldsymbol A_{s}\widehat{\mathbf u} := a(\widehat{\bf u}, \cdot; s), \quad B_{s}\widehat{\boldsymbol \theta}:= b \, ( \widehat{\theta}, \cdot; s)_{\Omega}, \quad \text{ and } C_{s}\widehat{\varphi}:= c (\widehat{\varphi}, \cdot; s)_{\Omega},
\]

and consider the function spaces

\begin{align*}
X := \,& {\mathbf H}^1( \Omega)  \times H^1(\Omega)  \times H^1_*(\Omega) \times H^{1/2}(\Gamma) \times H^{-1/2}(\Gamma), \\
X^{\prime}_0 :=\,&  ({\mathbf H}^1(\Omega))^{\prime} \times (H^1(\Omega) )^{\prime}  \times (H^1_*(\Omega))^{\prime}  \times H^{-1/2}(\Gamma) \times H^{1/2}(\Gamma).
\end{align*}

Then from \eqref{eq:3.13}, \eqref{eq:3.10b} and \eqref{eq:3.11}, we pose the nonlocal problem as

\noindent
{\bf The nonlocal boundary  problem.}\\
 {\it For problem data} $(\widehat{d}_1, \widehat{d}_2, \widehat{d}_3, \widehat{d}_4, \widehat{d}_5)  \in X ^\prime$, {\it given by} 
 
 \begin{equation*}
\begin{array}{lll}
\widehat{d}_1 = -  {\gamma^-}^{\prime}(\gamma^+\widehat{p}^{\, inc}{\mathbf n}), & \widehat{d}_2 =  {\gamma^-}^{\prime} ({\Theta}_0^{-1} \widehat{f}_{\theta}), & \widehat{d}_3 = - {\gamma^-}^{\prime}\widehat{f}_{\bf D},\\
 \widehat{d}_4 =  (\rho_f)^{-1} {\partial_n^+ \widehat{p}^{inc}}, & \widehat{d}_5 = 0, & 
\end{array}
\end{equation*}
 
 {\it find functions $(\widehat {\mathbf u},   \widehat{ \theta} , \widehat{\varphi}, \widehat{\phi}, \widehat{\lambda} ) \in X $ satisfying }

\begin{equation}
\label{eq:3.17}
\mathbb A(s)  \left(   \widehat{ \mathbf{u} },  \widehat{ \theta} , \widehat{\varphi} ,   \widehat{\phi} ,  \widehat{\lambda} \right)^{\top} =   
\left( \widehat{d}_1,  \widehat{d}_2, \widehat{d}_3, \widehat{d}_4
,  \widehat{d}_5 \right)^{\top}
\end{equation}

\textit{with}

\begin{equation} 
 \label{eq:3.17b}
\mathbb A(s) 
\begin{pmatrix}
     \widehat{ \mathbf{u} }  \\[2mm]
     \widehat{ \theta} \\[2mm]
     \widehat{\varphi} \\[2mm]
      \widehat{\phi} \\[2mm]
      \widehat{\lambda} \\[2mm]
\end{pmatrix} \!:=\!
\left (\!  
 \begin{matrix}
\boldsymbol A_{s} & -\zeta~(\nabla\cdot)^{\prime} &  \boldsymbol\varepsilon^\prime\,{\bf e}^\top \nabla & \gamma_n^{-^{\prime}}  &  0 \\[2mm]
  s \, \zeta \nabla\cdot & {T_0}^{-1}  B_{s} & - s \, {\bf p} \cdot \nabla   & 0 & 0 \\[2mm]
- \nabla^{{\prime}}  {\bf e}\, \boldsymbol\varepsilon  & - \nabla^{\prime}{\bf p} &  \epsilon C_{s} & 0 & 0 \\[2mm]
  - s^2 \gamma_n^{-}  & 0 & 0 & \! {\rho_f}^{-1}   W(s)  & \! - {\rho_f}^{-1}  ( \frac{1}{2} I -  K (s) )'\! \\[2mm]
 0 & 0 & 0 & \! \frac{1}{2}I - K(s) \! & V(s) 
 \end{matrix}\! 
\right ) \begin{pmatrix}
     \widehat{ \mathbf{u} }  \\[2mm]
     \widehat{ \theta} \\[2mm]
     \widehat{\varphi} \\[2mm]
      \widehat{\phi} \\[2mm]
      \widehat{\lambda} \\[2mm]
\end{pmatrix}.
 \end{equation}

In the next section we will show that this problem is in fact well-posed.
%
 \section{Variational solutions }\label{sec:VariationalSolutions} 
 %
We are interested in seeking variational solutions of the nonlocal  boundary problem \eqref{eq:3.17}  in the transformed domain.  To this end we need some additional preliminary results and definitions. We begin with the norms: 

\begin{alignat}{2}
\nonumber
\triple{\widehat{\mathbf u}}_{|s|, \Omega}^2 :=\,&( \bm{ \sigma} (\widehat {\mathbf u}),  \bm{\widehat{\varepsilon} } (\bar {\widehat{\mathbf u} } ) )_{\Omega} +  \rho_{e} \| \, |s| \, \widehat{\mathbf u }\|^2_{\Omega} \quad && \widehat{\mathbf u} \in {\mathbf H}^1(\Omega), 
 \\
 \nonumber
\triple{\widehat{\theta}}^2_{|s|, \Omega} :=\,& \| \nabla \widehat{\theta} \|^2_{\Omega}  + c_{\varepsilon} ^{-1}
 \| \sqrt{ |s|} \; \widehat{\theta} \|_{\Omega}^2 \quad&& \widehat{\theta} \in H^1(\Omega),\\
 \label{eq:4.3}
 \triple{\widehat{\varphi}}^2_{ 1,\,  \Omega} := \,&  \| \nabla \widehat{\varphi} \|^2_{\Omega}  \quad&& \widehat{\varphi }\in 
H^1_*(\Omega), \\
\nonumber
\triple{\widehat{p} }^2_{|s|,\Omega^c}:=\,& \| \nabla \widehat{p} \|^2_{\Omega^c} +  c^{-2} \| \, |s| \, \widehat{p } \|^2_{\Omega^c} \qquad &&\widehat{ p} \in H^1(\Omega^c).
\end{alignat}

For $\widehat{\varphi} \in 
H^1_*(\Omega) $,  we see that $ \| \nabla \widehat{\varphi} \|^2_{\Omega}= 0 $ if and only if $ \widehat{\varphi} = 0$. Hence, \eqref{eq:4.3} indeed defines a norm in $H^1_*(\Omega) $ (see Hsiao and Wendland \cite  [Lemma 5.2.5, p.255] {HsWe:2008} ). 

We will define  $\sigma:= \mathrm{Re}\,s$ and $ \underline{\sigma}:= \min\{1, \sigma\}$. With this notation, it is not hard to verify that

\[
\underline{\sigma} \leq \min\{1, |s|\},\quad \mbox{and} \quad  \max\{1, |s|\} \underline{\sigma}
 \leq |s|,~ \;\forall  s \in \mathbb{C}_+.
\]

Using these relations, it is possible to prove the following inequalities relating the energy norms defined above

 \begin{gather}
\underline{\sigma }\triple{\widehat{\mathbf u}}_{1, \Omega} \leq \triple{\widehat{ \mathbf u}}_{|s|, \Omega}\leq
\frac{|s|}{\underline{\sigma}} \triple{\widehat{\mathbf u}}_{1, \Omega} ,\label{eq:4.5}\\
\sqrt{\underline{\sigma}}\triple{\widehat{\theta}}_{1, \Omega} \leq \triple{\widehat{\theta}}_{|s|, \Omega_+}  \leq 
\sqrt{\frac{|s|}{\underline {\sigma} } }\triple{\widehat{\theta}}_{1,\Omega},   \label{eq:4.6} \\
\underline{\sigma } \triple{\widehat{p}}_{1, \Omega_+} \leq \triple{\widehat{p}}_{|s|, \Omega_+}\leq
\frac{|s|}{\underline{\sigma}} \triple{\widehat{p}}_{1, \Omega_+}. \label{eq:4.7}
 \end{gather}
These relations will be used heavily when estimating the norms of the solutions in terms of the Laplace parameter $s$ and its real part $\sigma$. The norms $\triple{\cdot }_{1, \Omega} $ and $ \triple{\cdot}_{1, \Omega^c}$ are respectively equivalent to $\| \cdot \|_{H^1(\Omega) } $ and $ \| \cdot \|_{H^1(\Omega^c) } $. An application of Korn's second inequality \cite{Fi:1972} shows that, for a vector-valued function ${ \widehat{ \mathbf u}} $, the energy norm $\triple{\cdot }_{1, \Omega}$ is also equivalent to the standard $\mathbf{H}^1(\Omega)$ norm.  

Now, given a vector of solutions $( \widehat{\bf u}, \widehat{\theta},\widehat{ \varphi}, \widehat{\phi} , \widehat{\lambda} )$ to 
\eqref{eq:3.17}, by defining
 
\[
\widehat{p}(s)  = D(s) \widehat{\phi} - S(s)\widehat{\lambda} \quad \mbox{in} \quad \mathbb{R}^3\setminus\Gamma,
\]

then $\widehat{p} \in H^1(\mathbb{R}^3\setminus \Gamma) $ is the unique solution of the transmission problem :  

\begin{alignat}{6}
\label{eq:4.10}
 - \Delta \widehat{p}  + \frac{s^2}{c^2}\,  \widehat{p}(s) =\,& 0  &&\quad \mbox{in} \quad \mathbb{R}^3\setminus \Gamma, \\
 \nonumber
\jump{\gamma \widehat{p}} =\,& \widehat{\phi} \in H^{1/2}(\Gamma) &&\quad \mbox{on} \quad \Gamma,  \\
\nonumber
\jump{\partial_n \widehat{p}}=\,& \widehat{\lambda} \in H^{-1/2}(\Gamma)  && \quad \mbox{on} \quad \Gamma, 
\end{alignat}

where the symbol $\jump{\cdot}$ denotes the "jump" relations of a function across $\Gamma$. More specifically we have 

\[
\jump{\gamma \widehat{p}}  := (\widehat{p}^+ - \widehat{p}^-),  \quad \mbox{and} \quad \jump{\partial_n \widehat{p} } : = (\partial^+_n \widehat{p} - \partial^-_n \widehat{p}).
\]

We remark that  in the present case no radiation condition is needed to ensure uniquness because of Huygen's principle. In terms of the jumps of $\widehat{p}$, the last two equations of \eqref{eq:3.17} are equivalent to 

 \begin{alignat}{5}
 \label{eq:4.12} 
 -s^2 \gamma_n^- \widehat{{\bf u}} -  \frac{1}{\rho_f}  \partial_n^+ \widehat{p} & =  \frac{1}{\rho_f}   \widehat {d}_4 \quad& \mbox{on} \quad \Gamma  \\
 \nonumber
 - \gamma^-  \widehat{p} & = \widehat{d}_5 \quad &\mbox{on} \quad \Gamma
 \end{alignat} 
 
Since $\widehat{d}_5 = 0$, we conclude that $\widehat{p}$ satisfies the homogeneous Dirichlet problem for \eqref{eq:4.10} in $\Omega$ and, by uniqueness, it must follow that $ \widehat{p} \equiv 0$ in $\overline{\Omega}$. As a consequence we have the following relations between the unknown densities and the Cauchy data

\begin{equation}
\jump{\gamma\widehat{p}} = \gamma ^+ \widehat{p} = \widehat{\phi}  \quad \mbox{and} \quad \jump{\partial_n \widehat{p}} =  \partial_n^+\widehat{p} = \widehat{ \lambda}. \label{eq:4.14} 
\end{equation} 

On the other hand, the transmission condition  \eqref{eq:4.12} is closely related to the variational equation of \eqref{eq:4.10}

\begin{align*} 
- \langle \partial^+_n \widehat{p}, \overline{\gamma^+\widehat{q}} \rangle_{\Gamma}  &= \int_{\Omega^c} (\nabla \widehat{p} \cdot \overline{\nabla \widehat{q} }+ (s/c)^2 \widehat{p} \,\overline{ \widehat{q}} \; ) dx\\
&=: d_{\Omega^c} (\widehat{p}, \widehat{q}; s)\\
&=: (D_s \widehat{p}, \widehat{q} )_{\Omega^c},
\end{align*}

where the domain of integration for the sesquilinear form $d_{\Omega^c} (\widehat{p}, \widehat{q}; s)$ and the associated operator $D_s$, has been indicated explicitly in the definition. Now, using \eqref{eq:4.12} we arrive at

\begin{equation*}
-s^2 \langle \gamma^- \widehat{{\bf u}}, \overline{\gamma^+ \widehat{q}} \; {\bf n} \rangle_{\Gamma} + \frac{1}{\rho_f} (D_s \widehat{p}, \overline{\widehat{q}} ) _{\Omega^c} = \langle \widehat{d}_4 , \overline{\gamma^+\widehat{q}} \rangle_{\Gamma} .
\end{equation*} 

Combining the above equality with the weak formulations of the first three equations in \eqref{eq:3.17}, we can formulate an equivalent variational problem. We will first introduce the space $ {\mathbb{H}} :=  {\mathbf H}^1(\Omega)  \times H^1(\Omega) \times H_*^ 1(\Omega) \times H^1(\Omega^c) $  and endow it with the norm:

\[
\|(\widehat {\bf u},\widehat{ \theta}, \widehat{\varphi}, \widehat{p})\|_{\mathbb H} := \left(  \, \triple{\widehat{\bf u} }^2_{ 1, \Omega} +   \triple{\widehat{\theta} }^2_{1, \Omega}  + \| \widehat{\varphi }\|^2_{\Omega}   +  \triple {\widehat{ p} }^2_{1, \Omega^c} \right)^{1/2}.
\]

\noindent
{\bf The variational problem.} {\it Find $( \widehat{\bf u}, \widehat{\theta} , \widehat{\varphi}, \widehat{p } ) \in {\mathbb{H}} $ satisfying} 

\begin{equation}
\label{eq:4.16}
 \mathcal{A}( (\widehat {\bf u},\widehat{ \theta}, \widehat{\varphi}, \widehat{p}), (\widehat{\bf v}, \widehat{\vartheta}, \widehat{\psi}, \widehat{q}); s) 
 = \ell_d((\widehat{\bf v}, \widehat{\vartheta}, \widehat{\psi}, \widehat{q})), \quad \forall (\widehat{\bf v}, \widehat{\vartheta}, \widehat{\psi}, \widehat{q}) \in {\mathbb{H}} 
\end{equation} 

where the sesquilinear form on the left hand side of the equation is defined by
 
\begin{alignat*}{6}
 \mathcal{A}( (\widehat {\bf u},\widehat{ \theta}, \widehat{\varphi}, \widehat{p}), (\widehat{\bf v}, \widehat{\vartheta}, \widehat{\psi}, \widehat{q}); s): = \,&
(\boldsymbol A_s \, \widehat{\bf u},\overline{ \widehat{\bf v} })_{\Omega} - \zeta ( \widehat{\theta}, \nabla\cdot\overline{\widehat{\bf v}})_{\Omega} +  ( \nabla \widehat{\varphi}, {\bf e}\, \bm{\varepsilon}(\overline{ \widehat{\bf v}}) )_{\Omega} + \langle \gamma^+\widehat{p} \; {\bf n} , \gamma^{-}\overline{\widehat{\bf v} }\rangle_{\Gamma} \\
&+ s \zeta (\nabla\cdot\widehat{\bf u} , \overline{\widehat{\vartheta}} )_{\Omega} + T_0^{-1} ( B_s \widehat{\theta}, \overline{\widehat{\vartheta}} )_{\Omega} - s ({\bf p} \cdot \nabla \widehat{\varphi} , \overline{\vartheta} )_{\Omega}\\ 
& -( {\bf e}\,  \bm{\varepsilon} (\widehat{\bf u}), \nabla \overline{\widehat{\psi}} )_{\Omega} - ( {\bf p}\,  \widehat{ \theta} , \nabla \overline{\widehat{\psi}} )_{\Omega} + \epsilon ( C_s \widehat{\varphi} , \overline{\widehat{\psi}})_{\Omega} \\ 
& -s^2 \langle \gamma^- \widehat{{\bf u}}, \overline{\gamma^+ \widehat{q}} \; {\bf n} \rangle_{\Gamma} + \frac{1}{\rho_f} (D_s \widehat{p}, \overline{\widehat{q}} ) _{\Omega^c} 
\end{alignat*}
 
for $(\widehat{\bf v}, \widehat{\vartheta}, \widehat{\psi}, \widehat{q}) \in {\mathbb{H}}$. The bounded linear functional on the right hand side is defined by 

\[
 \ell_d((\widehat{\bf v}, \widehat{\vartheta}, \widehat{\psi}, \widehat{q}) ): =
 (\widehat{d}_1, \overline{\widehat {\bf v}} )_{\Omega} + (\widehat{d}_2, \overline{\widehat{\vartheta}})_{\Omega} + (\widehat{d}_3, \overline{\widehat{\psi}})_{\Omega} + \langle \widehat{d}_4 , \overline{\gamma^+\widehat{q}} \rangle_{\Gamma},
\]

for all tests $(\widehat{\mathbf v},\widehat{\vartheta},\widehat{\psi},\widehat{q})\in \mathbb H$. By construction, this variational problem is equivalent to the transmission problem \eqref{eq:3.1a} through \eqref{eq:3.4} which in turn is equivalent to \eqref{eq:3.17}. Consequently, it suffices to show the existence of a solution of \eqref{eq:4.16} to guarantee that \eqref{eq:3.17} is indeed solvable. We now present the following basic existence and uniqueness results. 

\begin{Theorem}\label{th:4.1}
Under the assumption of the constant pyroelectric moduli vector vector ${\bf p} $ satisfying the constraint  

\[
\|{\bf p}\|_{\mathbb{R}^3} <  \min \{  \epsilon ,  \frac{c_{\varepsilon}} {T_0} \} ,  
\] 

the variational problem {\em \eqref{eq:4.16}  has a unique solution}  $ ( \widehat{\bf u}, \widehat{\theta,}  \widehat{\varphi} ,\widehat{ p}) \in {\mathbb{H}} $. Moreover, the following estimate holds:
 
\begin{align} \label{eq:4.18}
\|(\widehat {\bf u},\widehat{ \theta}, \widehat{\varphi}, \widehat{p})\|_{\mathbb H} &\leq c_0 \; \frac{ |s|^3} {\sigma {\underline{\sigma}}^6} \; \| (\widehat{d}_1, \widehat{d}_2, \widehat{d}_3, \widehat{d}_4) \|_{\mathbb H^{\prime}}. 
\end{align}
  
Here and in the sequel, $c_0>0$ will denote a constant that may depend on $\rho_f, T_0, c_{\varepsilon}, \epsilon, {\bf p}$ only. 
\end{Theorem}
\begin{proof}
Let $ \mathcal{A}( (\widehat {\bf u},\widehat{ \theta}, \widehat{\varphi}, \widehat{p}), (\widehat{\bf v}, \widehat{\vartheta}, \widehat{\psi}, \widehat{q}); s) $ be the sesquilinear  form defined by the variational equation \eqref{eq:4.16}.  We first show that $\mathcal{A} $  is continuous.
It is easy to verify that 

\begin{align*}
| (\boldsymbol{A}_s \, \widehat{\bf u},\overline{ \widehat{\bf v} })_{\Omega}\! + \! 
 T_0^{-1} ( B_s \widehat{\theta}, \overline{\widehat{\vartheta}} )_{\Omega} 
\! + \! \epsilon ( C_s \widehat{\varphi} , \overline{\widehat{\psi}})_{\Omega}
\! +  \! \frac{1}{\rho_f}\,& (D_s \widehat{p}, \overline{\widehat{q}} ) _{\Omega^c} | \! \leq \\
 & m_1 \left(\!\frac{|s|}{\underline{\sigma}}\! \right)^2 \! \|\!({\bf u}, \theta, \varphi, p) \!\|_{ {\mathbb{H}} } \|\!(\widehat{\bf v}, \widehat {\vartheta}, \widehat{\psi}, \widehat{q})\!\|_{\mathbb{H} }. 
\end{align*}

The remaining terms in $\mathcal{A}( (\widehat {\bf u},\widehat{ \theta}, \widehat{\varphi}, \widehat{p}), (v, \vartheta, \psi, q ) )$ can be bounded easily using the Cauchy-Schwartz inequality, Poicar\'e's inequality in $H^1_*(\Omega)$, the trace theorem, and the estimate  

\[
|\left(\nabla \varphi, {\bf e}\,  \bm{\varepsilon}( {\bf v})\right)_{\Omega} | \leq \mathrm e_{max} \|\nabla \varphi\|_{\Omega}\;  \|\nabla\cdot{\bf v} \|_{\Omega}.
\]

This leads to the continuity estimate 

\[ 
 \mathcal{A}( (\widehat {\bf u},\widehat{ \theta}, \widehat{\varphi}, \widehat{p}), (\widehat{\bf v}, \widehat{\vartheta}, \widehat{\psi}, \widehat{q}); s )
  \leq (m_1 +  m_2)  \left(\frac{|s|} {\underline{\sigma}}\right)^2  \| ({\bf u}, \theta, \varphi, p)\|_{ {\mathbb{H}} } \, \| (\widehat{\bf v}, \widehat {\vartheta}, \widehat{\psi}, \widehat{q} ) \|_{{\mathbb{H}}}.
\]

Here  $m_1$ and $m_2$  are constants depending only upon the physical parameters $ \zeta, \Theta_0, {\bf p}, \epsilon$, and $\mathrm e_{max}= \max \{ \mathrm e_{ijk}, i, j, k = 1\cdots 3\} $.

We now introduce the scaling factor

\begin{equation} \label{eq:4.20}
Z(s):= \left (\begin{matrix}
\bar{s}&0 &0&0 \\
0& 1&0 & 0 \\
0&0& s & 0\\
0&0&0& \bar{s} / {|s|^2}
\end{matrix} \right ),
\end{equation}

and note that, for $ (\widehat {\bf u},\widehat{ \theta}, \widehat{\varphi}, \widehat{p}) \in {{\mathbb{H}}}$, we have

\begin{align*}
 \mathrm{Re} \Big( &\,\bar{s}\left(- \zeta ( \widehat{\theta}, \nabla\cdot\overline{\widehat{\bf u}})_{\Omega} + ({\bf e}^{\top} \nabla \widehat{\varphi}, \bm{\varepsilon}(\overline{ \widehat{\bf u}}) )_{\Omega} + \langle \gamma^+\widehat{p} \; {\bf n} , \gamma^{-}\overline{\widehat{\bf u} }\rangle_{\Gamma} \right) \\
& + s \left(\zeta (\nabla\cdot\widehat{\bf u} , \overline{\widehat{\theta}} )_{\Omega} 
-( {\bf e}\,  \bm{\varepsilon} (\widehat{\bf u}), \nabla \overline{\widehat{\varphi} })_{\Omega}\right) 
- \left(\bar{s}/ {|s|^2}\right)s^2 \langle \gamma^- \widehat{{\bf u}}, \overline{\gamma^+ \widehat{p}} \; {\bf n} \rangle_{\Gamma}  \Big)
= 0.
 \end{align*}

Therefore, it follows that

\begin{align}
\nonumber
\mathrm{Re} \left(\!Z(s)  \mathcal{A}( (\widehat {\bf u},\widehat{ \theta}, \widehat{\varphi}, \widehat{p}), (\widehat {\bf u},\widehat{ \theta}, \widehat{\varphi}, \widehat{p}
); s ) \!\right)\!
 = \, & \mathrm{Re} \left(
 \bar{s}( \boldsymbol{A}_s \, \widehat{\bf u},\overline{ \widehat{\bf u} })_{\Omega}
\! + \! T_0^{-1} ( B_s \widehat{\theta}, \overline{\widehat{\theta}} )_{\Omega} \right.\\
\nonumber
&  
\; \quad -  s \left( ({\bf p} \cdot \nabla \widehat{\varphi},  \overline{\widehat{\theta} })_{\Omega}\
\! - \! ( {\bf p}\,  \widehat{ \theta} , \nabla \overline{\widehat{\varphi}} )_{\Omega}
\! + \! \epsilon ( C_s \widehat{\varphi} , \overline{\widehat{\varphi}})_{\Omega} \right) \\
\label{eq:4.21}
& \left.
\; \quad +(\bar{s} / |s|^2)\rho_f^{-1}(D_s \widehat{p}, \overline{\widehat{p}} ) _{\Omega^c} \right).
\end{align}
 
By setting to zero some of the entries of $(\widehat {\bf u},\widehat{ \theta}, \widehat{\varphi}, \widehat{p})$ in the right hand side of \eqref{eq:4.21}, it is possible to derive the following

 \begin{equation} 
 \left.
\begin{aligned}
\mathrm{Re} \left( \bar{s} ( \boldsymbol{A}_s \, \widehat{\bf u},\overline{ \widehat{\bf u} } )_{\Omega} ) \right) =\,& \sigma \, \triple{\widehat{\bf u} }^2_{ |s|, \Omega} \\
\mathrm{Re}\left(T_0^{-1}  ( B_s \widehat{\theta}, \overline{ \widehat{\theta} } )_{\Omega} \right)  =\,& T_0^{-1}  \left( \| \nabla \widehat{\theta} \|^2_{\Omega }+ c_{\varepsilon} \sigma \;  \| \widehat{\theta }\|^2 \right) \\
\mathrm{Re} \left(- s \left(({\bf p} \cdot \nabla \widehat{\varphi}, \overline{\widehat{\theta} })_{\Omega} + 
 ({\bf p}\,  \widehat{ \theta} , \nabla \overline{\widehat{\varphi}} )_{\Omega} \right)\right) \geq\,&  -  \sigma  \| {\bf p} \|_{\mathbb{R}^3}  \big(\| \nabla \varphi \|_{\Omega} + \| \theta \|_{\Omega} \big) \\
 \mathrm{Re} \left( s \epsilon ( C_s \widehat{\varphi} , \overline{\widehat{\varphi}})_{\Omega} \right)  =\,&  \sigma \, \epsilon\,  \| \nabla \varphi \|^2_{\Omega} \\
\mathrm{Re} \left( (\bar{s} / |s|^2) \rho_f^{-1} (D_s \widehat{p}, \overline{\widehat{p}} ) _{\Omega^c} \right)  =\,&  (\sigma/|s|^2)\rho_f^{-1} \triple{ p}^2_{|s|, \Omega^c}  
 \end{aligned}
 \right\}  \label{eq:4.22}
 \end{equation} 
 
From \eqref{eq:4.22} and \eqref{eq:4.21}, it follows that
 
\begin{equation} 
 \label{eq:4.23}
\mathrm{Re} \left(\!Z(s) \mathcal{A}( (\widehat {\bf u},\widehat{ \theta}, \widehat{\varphi}, \widehat{p}),  (\widehat {\bf u},\widehat{ \theta}, \widehat{\varphi}, \widehat{p})
; s)\!\right) \! \geq \! 
\frac{ \sigma\; \underline{\sigma}^2}  {|s|^2 }\left( \!\triple{\widehat{\bf u} }^2_{ |s|, \Omega} \! + \! c_1 \triple{\widehat{\theta}}^2_{|s|, \Omega} \! + \! c_2 \| \widehat{\varphi }\|^2_{\Omega} \! + \! \triple {\widehat{p}}^2_{|s|, \Omega^c}  \! \right),
\end{equation} 

where $ c_1  =  c^{-1} _{\varepsilon}( c_{\varepsilon}{ \Theta^{-1}_0} - \| \widehat{\bf p } \|_{\mathbb{R}^3})  > 0 $ and $ c_2  = (\epsilon - \| {\bf p}  \|_{\mathbb{R}^3} ) > 0 $. Alternatively,  in view of \eqref{eq:4.5} - \eqref{eq:4.7}, we have
 
\[
| \mathcal{A}( (\widehat {\bf u},\widehat{ \theta}, \widehat{\varphi}, \widehat{p}),  (\widehat {\bf u},\widehat{ \theta}, \widehat{\varphi}, \widehat{p}) 
; s)|
\;  \geq  \alpha_0 \; \frac{ \sigma\; \underline{\sigma}^6}  {|s|^3 } \| (\widehat{\bf u}, \widehat{\theta} , \widehat{\varphi}, \widehat{ p} ) \|^2_{{\mathbb{H}}}
\] 

where $\alpha_0 > 0 $ is a constant independent of $\sigma$, and  $|s|$.  Hence, by the Lax-Milgram  lemma, there exists a unique solution of the variational problem \eqref{eq:4.16}. 

Having shown that the problem is uniquely solvable, the stability estimate \eqref{eq:4.18} can be derived from \eqref{eq:4.23} and \eqref{eq:4.16} as we show next 

\begin{align*}
 \frac{ \sigma\; \underline{\sigma}^2}  {|s|^2 } \left(\!\triple{\widehat{\bf u} }^2_{ |s|, \Omega} \! + \! c_1 \triple{\widehat{\theta}}^2_{|s|, \Omega} \! +  c_2 \right.\, & \left. \| \widehat{\varphi} \|^2_{\Omega} \! + \! \triple {\widehat{p}} ^2_{|s|, \Omega^c} \! \right) \!
  \leq \! \mathrm{Re} \left( \! Z(s) \mathcal{A}( (\widehat {\bf u},\widehat{ \theta}, \widehat{\varphi}, \widehat{p}),  (\widehat {\bf u},\widehat{ \theta}, \widehat{\varphi}, \widehat{p}) ; s) \! \right)  \\
  \leq\,& \left| \bar{s}(\widehat{d}_1, \widehat{\bf u})_{\Omega} + (\widehat{d}_2, \widehat{\theta} )_{\Omega} + s (\widehat{d}_3,\widehat{ \varphi})_{\Omega} + \bar{s} /|s|^2 \langle \widehat{d}_4, \gamma^+ \widehat{p} \rangle_{\Gamma} \right|\\
\leq\,&  \frac{|s|}{\underline{\sigma}^2} \left(|(\widehat{d}_1, \widehat{\bf u} )_{\Omega}| + |(\widehat{d} _2, \widehat{\theta} )_{\Omega}| +  |(\widehat{d}_3, \varphi)_{\Omega}| + |\langle \widehat{d}_4, \gamma^+ \widehat{ p} \rangle_{\Gamma} | \right).
\end{align*}

Consequently, using the first inequality of the equivalences \eqref{eq:4.5} through \eqref{eq:4.7}, we have the estimate 

\begin{equation}
\label{eq:4.24}
 \left(\triple{\widehat{\bf u} }^2_{ |s|, \Omega} +   \triple{\widehat{\theta} }^2_{|s|, \Omega}  +  \| \widehat{\varphi }\|^2_{\Omega}   +  \triple {\widehat{p} }^2_{|s|, \Omega^c} \right)^{1/2} \leq c_0 \; \frac{ |s|^3} {\sigma {\underline{\sigma}}^5} \; \| (\widehat{d}_1, \widehat{d}_2, \widehat{d}_3, \widehat{d}_4, 0) \|_{X^{\prime}}. 
\end{equation} 

Here $c_0$  is a constant depending only on the physical parameters $\rho_f, T_0, c_{\varepsilon}, \epsilon, {\bf p} $. The desired estimate \eqref{eq:4.18} can then be easily derived by simplifying the right hand side of the expression above and applying \eqref{eq:4.5} through \eqref{eq:4.7} to the term on left hand side. 
\end{proof} 

The estimate \eqref{eq:4.24} will lead us to verify the invertibility of the operator matrix $\mathbb A(s)$  defined in \eqref{eq:3.17b}, as we now show.  

\begin{Theorem} \label{th:4.2}
 The operator $ \mathbb A(s): X \rightarrow  X^\prime_0 $ as defined in \eqref{eq:3.17b} is invertible. Moreover, we have the estimate:

\begin{equation} \label{eq:4.25}
\| \mathbb A^{-1}(s)|_{X_0^{\prime}} \|_{X^\prime,   X}  \leq c_0 \frac{|s|^{3+1/2}}{\sigma\underline\sigma^{6+1/2}}.
\end{equation}
  
\end{Theorem}

\begin{proof}
From \eqref{eq:4.14}, we see that 

\[
\jump{\gamma\widehat{p}} = \gamma ^+ \widehat{p} = \widehat{\phi}  \quad \mbox{and} \quad \jump{\partial_n \widehat{p}} =  \partial_n^+\widehat{p} = \widehat{ \lambda}.  
\] 

From which it can be shown that (see, e.g. \cite{HSW:2013}): 

\begin{equation}
\label{eq:4.26}
\|\widehat {\phi } \|^2_{H^{1/2}(\Gamma)} = \|\gamma^+ \widehat{ p} \|_{H^{1/2}(\Gamma)}^2 \leq c_1 \,  \triple{\widehat{p} }^2_{1, \Omega^c} \leq c_1 \frac{1}{\underline{\sigma}^2}  \triple{\widehat{p}}^2_{|s|, \Omega^c}
\end{equation}

Similarly, we have

\[
|\! \langle \widehat{\lambda}, \widehat{q}^+\rangle \! | \! = \! | \!\langle \partial_n^+\widehat{ p} , \widehat{q}^+\rangle \! | \! = \! |a_{s, \Omega^c}(\widehat{p}, \widehat{q} )| \! 
\leq \! \triple{\widehat{p} }_{|s|, \Omega^c} \triple{\widehat{q}}_{|s|, \Omega^c} \!
\leq \! c_2 \sqrt{ |s|/\underline{\sigma} }\,\triple{\widehat{p}}_{|s|, \Omega^c}  \| \widehat{q}^+ \! \|_{H^{1/2}(\Gamma)} 
\]

which implies

\begin{equation}
\label{eq:4.28}
\| \widehat{\lambda} \|_{H^{-1/2}(\Gamma)} = \sup_{0 \ne q^+ \in H^{-1/2} (\Gamma)}  \frac{  | \langle \partial_n^+ \widehat{p}, \widehat{ q}^+ \rangle _{\Gamma} | }
{ \; \| \widehat{q}^+ \|_{H^{1/2} (\Gamma) } }  \leq \, c_2 \sqrt{ |s|/\underline{\sigma} }\,\triple{\widehat{p}}_{|s|, \Omega^c}. 
\end{equation} 

Above, Bamberger and Ha-Duong's optimal lifting \cite{BaHa:1986a, BaHa:1986b} has been used to bound the norm  $ \triple{\widehat{q}}_{|s|, \Omega^c} $ by  $ \| \widehat{q}^+\|_{H^{1/2}(\Gamma)} $  in \eqref{eq:4.25}.  Then \eqref{eq:4.26} and \eqref{eq:4.28} yield  
 the estimates
\begin{equation}
\frac{1}{2} \Big(\frac{1}{c_1} \underline{\sigma}^2  \| \widehat{\phi }\|^2_{H^{1/2}(\Gamma)} + 
 \frac {\underline\sigma}{c_2^2|s|}  \| \widehat{\lambda} \|^2_{H^{-1/2}(\Gamma)}\Big)
 \leq \triple{ \widehat{p} }^2_{|s|, \Omega_+}.\label{eq:4.29}
 \end{equation}

As a consequence of \eqref{eq:4.24}, it follows  from \eqref{eq:4.29} that 

\begin{align*}
 \Big(\underline{\sigma}^2 \triple{\widehat{\bf u} }^2_{ 1, \Omega} +  \underline{\sigma} \triple{\widehat{ \theta} }^2_{1, \Omega} 
 +  \| \widehat{\varphi}  \|^2_{\Omega} & +
 \frac{1}{2} \Big(\frac{\underline{\sigma}^2 }{c_1} \|\ \widehat{\phi}  \|^2_{H^{1/2}(\Gamma)} + 
 \frac {\underline\sigma}{c_2^2|s|}  \| \widehat{ \lambda} \|^2_{H^{-1/2}(\Gamma)}\Big)\Big)^{1/2}\\
&\leq c_0 \; \frac{ |s|^3} {\sigma {\underline{\sigma}}^5} \; \| (\widehat{d}_1, \widehat{d}_2, \widehat{d}_3, \widehat{d}_4, \widehat{d}_5) \|_{X_0^{\prime}}
\end{align*} 

which  implies

\begin{align*}
\Big(\!\triple{\widehat{ \mathbf u}}^2_{1, \Omega_-} \! + \! \triple{\widehat{\theta}}^2_{1, \Omega_-} \! + \! \| \widehat{\varphi} \|^2_{H^1_*(\Omega)}
\! + \! \| \widehat{\phi } \|^2_{H^{1/2}(\Gamma)} \! + \,& \| \widehat{ \lambda} \|^2_{H^{-1/2}(\Gamma)}\Big)^{1/2}\\
 & \leq  c_0 \frac{|s|^{3+1/2} }{\sigma\underline{\sigma}^{6+1/2}} \|(\widehat{d}_1, \widehat{d}_2, \widehat{d}_3, \widehat{d}_4, \widehat{d}_5)\|_{X_0^\prime}. 
\end{align*}
\end{proof}

\section{Results in the time domain}\label{sec:TimeDomain} 
%
Having established the properties of the operators and solutions to our problem in the Laplace domain, we can now return to the time domain and establish analogue results. In order to state the result that will allow us to transfer our previous analysis in the Laplace domain back in to the time domain, following \cite{LaSa:2009b}, let us first define a class of admissible symbols.

The following definition, and the proposition following immediately after it (an improved version of \cite[Proposition 3.2.2]{Sayas:2016}, and \cite{Sayas:2016errata}) will be used to transform the Laplace-domain bounds into time-domain statements.\\[2mm]

\noindent
{\bf A class of  admissible symbols}: Let  $\mathbb X$ and $\mathbb Y$ be Banach spaces and  $\mathcal{B}(\mathbb X, \mathbb Y)$ be the set of bounded linear operators from $\mathbb X$ to $\mathbb Y$. An operator-valued analytic function $A : \mathbb{C}_+ \rightarrow \mathcal{B}(\mathbb X, \mathbb Y)$ is said to belong to the class $ \mathcal{A} (\mu, \mathcal{B}(\mathbb X, \mathbb Y))$, if there exists a real number $\mu$ such that 

\[
\|A(s)\|_{\mathbb X,\mathbb Y} \le C_A\left(\mathrm{Re} (s)\right) |s|^{\mu} \quad \mbox{for}\quad s \in \mathbb{C}_+ ,
\]

where the function $C_A : (0, \infty) \rightarrow (0, \infty) $ is non-increasing and satisfies 

\[
C_A(\sigma) \le \frac{ c}{\sigma^m} , \quad \forall \quad \sigma \in ( 0, 1]
\]

for some  $m \ge 0$ and $c$ independent of $\sigma$.  

\begin{Proposition} {\em (\cite{Sayas:2016errata})} \label{pr:5.1}
Let $A = \mathcal{L}\{a\} \in \mathcal{A} (k + \alpha, \mathcal{B}(\mathbb X,\mathbb Y))$ with $\alpha\in [0, 1)$ and $k$ a non-negative integer.  If $ g \in \mathcal{C}^{k+1}(\mathbb{R}, \mathbb X)$ is causal and its derivative $g^{(k+2)}$ is integrable, then $a* g \in \mathcal{C}(\mathbb{R}, \mathbb Y)$ is causal and 

\[
\| (a*g)(t) \|_{\mathbb Y} \le 2^{\alpha} C_{\epsilon} (t) C_A (t^{-1}) \int_0^1 \|(\mathcal{P}_2g^{(k)})(\tau) \|_{\mathbb X} \; d\tau,
\]

where 

\[
C_{\epsilon} (t) := \frac{1}{2\sqrt{\pi}} \frac{\Gamma(\epsilon/2)}{\Gamma\left( (\epsilon+1)/2 \right) } \frac{t^{\epsilon}}{(1+ t)^{\epsilon}}, \qquad (\epsilon :=  1- \alpha \; \;  \mbox{and}\; \;  \mu = k +\alpha)
\]

and 

\[
(\mathcal{P}_2g) (t) =  g + 2\dot{g} + \ddot{g}.
\]
\end{Proposition}

The results proven in Section \ref{sec:VariationalSolutions}---specifically the bounds obtained in terms of the Laplace parameter $s$ and its real part $\sigma$---will now allow us to show that the operators involved belong precisely to one such class of symbols. 
  
 We begin with the results in Theorem \ref{th:4.1} and  from \eqref{eq:4.18}, we may write 
 
 \begin{gather}
 \nonumber
 (\widehat{\bf u}, \widehat{\theta}, \widehat{\varphi}, \widehat{p} )^{\top} = A(s) (\widehat{d}_1, \widehat{d}_2, \widehat{d}_3, \widehat{d}_4, 0)^{\top},\\
 \label{eq:5.2}
  \| A(s)|_{X_0^{\prime}} \|_{X^{\prime}, {\mathbb{H}}}  \leq C_{A}\,  \frac{|s|^3}{\sigma \underline{\sigma}^6}.
 \end{gather}
 
 Hence,  $A(s) \in \mathcal{A} ( 3, \mathcal{B}( X_0^{\prime}, {\mathbb{H}}))$, and 
 
 \begin{align*}
 ({\bf u}, \theta, \varphi, p )^{\top}  &= \mathcal{L}^{-1} \Big\{  A(s) \,  (\widehat{d}_1, \widehat{d}_2, \widehat{d}_3, \widehat{d}_4, 0)^{\top} \Big\}
\\
&= \mathcal{L}^{-1}\{A(s)\} * \mathcal{L}^{-1}\{(\widehat{d}_1, \widehat{d}_2, \widehat{d}_3, \widehat{d}_4, 0)^{\top}\}\\
& = ( \mathcal{L}^{-1}\{A\} * {\bf D}) (t)\\
& =:( a* g)(t) \quad \mbox{ according to  Proposition \ref{pr:5.1} }. 
 \end{align*}
 
From the estimate of $A(s) $ in  \eqref{eq:5.2} , we have  

\[
\mu = k +\alpha = 3 \quad \mbox{implies }\quad   k= 3, \;  \alpha = 0 \quad \mbox{and} \quad \varepsilon = 1 - \alpha = 1.
\]

Thus, we have established  the following theorem. 

\begin{Theorem} \label{thm:5.2} 
Let \,  ${\mathbb{H}} :=  \mathbf {H}^1(\Omega)  \times H^1(\Omega) \times H_{*}^ 1(\Omega) \times H^1(\Omega^c)$. If

\[
{\bf D}(t):= \mathcal{L}^{-1}\left\{ ( \widehat{d}_1, \widehat{d}_2, \widehat{d}_3, \widehat{d}_4, 0 )^{\top} \right\}  \in
 \mathcal{C}^4( \mathbb{R},  X^{\prime}_0)
\]
 
 is causal and its derivative ${\bf D}^{(5)} $ is integrable,  then $( {\bf u}, \theta, \varphi, p)^{\top} \in \mathcal{C} (\mathbb{R} , {\mathbb{H}} )^{\top} $ is causal and
   
\[
 \|( \mathbf{u}, \theta, \varphi, p )^{\top }(t) \|_{{\mathbb{H}}}\leq c_0\;  \frac{t^2} {1 + t} \max \{ 1, t^6\} \int_0^t  \|(\mathcal{P}_2 {\bf D}^{(3)} ) (\tau) \| d\tau
\]

for some constant $c_0>0$, where $ (\mathcal{P}_2 {\bf D})(t)  = {\bf D} + 2\,  {\bf \dot{D}} + {\bf \ddot{D}} $.
\end{Theorem}

Similarly, from Theorem \ref{th:4.2}, we have 
 
\[
\left(   \widehat{ \mathbf{u} },  \widehat{ \theta} , \widehat{\varphi} ,   \widehat{\phi} ,  \widehat{\lambda} \right)^{\top} = \mathbb A^{-1} (s)   
\left( \widehat{d}_1,  \widehat{d}_2, \widehat{d}_3, \widehat{d}_4
,  \widehat{d}_5 \right)^{\top},
\]

from which, using \eqref{eq:4.25}, we infer that

\[
\| \mathbb A^{-1}(s) |_{X_0^{\prime}} \|_{X^\prime, \;   {\mathbb{H}}}  \leq c_0 \frac{|s|^{3+1/2}}{\sigma\underline\sigma^{6+1/2}} ,
\] 

hence $  \mathbb A^{-1}(s)  \in \mathcal{A} ( 3 \frac{1}{2},  \mathcal{B}( X_0^{\prime}, {\mathbb{H}})) $. Applying Proposition \ref{pr:5.1} with

\[
\mu =  (k + \alpha ) = 3 \frac{1}{2} , \;    k= 3, \;  \alpha = 1/2, \;  \varepsilon = 1 - \alpha = 1/2,
\]

then yields the following theorem.

\begin{Theorem}
Let  $ X :=  \mathbf {H}^1(\Omega)  \times H^1(\Omega) \times H_{*}^ 1(\Omega) \times H^{1/2}(\Gamma)  \times H^{-1/2}(\Gamma) $. If 
 
\[
\mathbf{D}(t) := \mathcal{L}^{-1} \{ ( \widehat{d} _1, \widehat{d}_2, \widehat{d}_3, \widehat{d}_4, \widehat{d}_5 )^{\top} \}(t) \in \mathcal{C}^{(4)} (\mathbb{R} , X_0^{\prime}) 
\]

is causal, and its derivative  ${\bf D}^{(5)}$  is integrable,   then $({\bf u}, \theta, \varphi, \phi, \lambda) ^{\top} \in C( \mathbb{R}, X) $  is casual,  and  there holds the estimate

\begin{gather*}
\|( ({\bf u}, \theta, \varphi,  \phi, \lambda) ^{\top} (t)\|_X  \leq ~c_{1/2}~ \frac{t^{ 1 + 1/2 } }{(1+ t)^{1/2}} ~\max \{1, t^{6 + 1/2} \} \!\!\int_0^t
 \| ( \mathcal{P}_2\mathbf D^{(3)}) (\tau)\|_{X^{\prime} }\;  d \tau ,  \\[2mm]
  (\mathcal{P}_2 {\bf D})(t) : = {\bf D} + 2\,  {\bf \dot{D}} + {\bf \ddot{D}}.
 \end{gather*}

 for some constant $c_{1/2} > 0$. 
\end{Theorem}

In view of \eqref{eq:3.5} and the inverse of $\mathbb A(s)$, we see that $\widehat{\mathbf u}, \widehat{\theta}, \widehat{\varphi} $ and $\widehat{p}$ are simply solutions of the following system

\begin{equation} \label{eq:5.3}
 (\widehat{\bf u}, \widehat{\theta}, \widehat{\varphi}, \widehat{p} )^{\top} 
=   A_2(s) \circ \mathbb A^{-1} (s) \;
     (\widehat{d}_1, \widehat{d}_2, \widehat{d}_3, \widehat{d}_4, 0)^{\top},   
 \end{equation}
 
 where 

\begin{equation*}
 A_2(s)  =
 \left( \begin{matrix}  
      1& 0 & 0 & 0 & 0\\[3mm]
      0 & 1& 0 & 0  & 0 \\[3mm]
      0& 0 & 1 & 0  & 0 \\[3mm]
      0 & 0 & 0  & D(s) & - S(s)\\[3mm]
    \end{matrix}
    \right). 
\end{equation*}

As a consequence of Theorem \ref{thm:5.2}, $A_2(s)\circ  \mathbb A^{-1} (s)$ belongs to the class 
 $\mathcal{A}(3 + 1/2, \mathcal{B}(X_0^{\prime}, {\mathbb{H}}))$. However, we  may also compute the index of the matrix of operators $A_2(s)$.  For $ \widehat{\phi} \in H^{1/2}(\Gamma)  $, let  $\widehat{u} =  D(s)  \phi \in \mathbb{R}^3 \setminus \Gamma $, then 
 
 \begin{align*}
 \sigma \triple{\widehat{u}}^2_{|s|, \mathbb{R}^3 \setminus \Gamma} & = \mathrm{Re} \left( \overline{s} \; \left\langle W \widehat{\phi}, \overline{\widehat{\phi}} \right\rangle_{\Gamma} \right) \\
 &\leq |s|\; \| W \widehat{\phi} \|_{H^{-1/2} (\Gamma)} \; \| \widehat{\phi } \|_{H^{1/2}(\Gamma)} \\
 &\leq c_1 \left( \frac{|s|}{\underline{\sigma}}\right)^{1/2}  \; |s|  \, \triple{\widehat{u}}_{|s|, \mathbb{R}^3 \setminus \Gamma} \;  \| \widehat{ \phi} \|_{H^{1/2} (\Gamma) }
 \end{align*}
 
 Hence, from \eqref{eq:4.7}, we obtain 

\[
\| D(s) \widehat{\phi} \|_{ H^1(\mathbb{R}^3 \setminus \Gamma) } \leq c_1  \frac{ |s|^{3/2} } {\sigma \underline{
 \sigma} ^{3/2}} \|\widehat{\phi}\|_{H^{1/2}(\Gamma)},
\]
 
which implies $D(s) \in \mathcal{A} (3/2, \mathcal{B}( H^{1/2}(\Gamma), H^1(\mathbb{R}^3 \setminus \Gamma) ) $. Similarly, for $\widehat{\lambda} \in H^{-1/2}(\Gamma) $, if we set $\widehat {u} = S(s) \widehat{\lambda} $ in $\mathbb{R}^3 \setminus \Gamma$, then we may show that

\[
\| S(s) \widehat{\lambda} \|_{ H^1(\mathbb{R}^3 \setminus \Gamma) }   \leq c_2  \frac{|s|} {\sigma \underline{\sigma}^2 }
  \| \widehat{\lambda} \|_{ H^{-1/2} (\Gamma)}.
\]
 
That is, $ S(s)  \in  \mathcal{A} (1, \mathcal{B} ( H^{-1/2}(\Gamma), H^1(\mathbb{R}^3 \setminus \Gamma))$, and hence

\[
\| A_2(s) \|_{ X, {\mathbb{H}}}  \leq  c_3 \frac{ |s|^{1 + 1/2}} { \sigma \underline{\sigma} ^{ 2+1/2}}.
\]

Following \cite{LaSa:2009b}, if we apply the composition rule and make use of the estimate of $\mathbb A^{-1}(s)$ in \eqref{eq:4.25}, we find the  matrices of the operators in \eqref{eq:5.3} ended with an index  $\mu = (1 +1/2) +  (3 + 1/2) = 5$. However, this only gives an upper bound for the actual index of $A_2(s) \circ \mathbb A^{-1} (s)$  in \eqref{eq:5.3}. 
 
\section{Concluding remarks}\label{sec:Conclusions}
%
A few remarks should be in order. This paper is dealing with a time-dependent wave- thermopiezoelectric structure interaction problem by the time-dependent boundary-field equation approach. With the help of a appropriate scaling factor $Z(s)$ in \eqref{eq:4.20}, we are able to establish the existence and uniqueness of the solutions to the problem. For simplicity, in this paper we only impose natural boundary conditions for the corresponding partial differential equations involved in the the interior domain $\Omega$. Clearly, one may also impose mixed boundary conditions as well. Moreover, the results presented in this communication generalize those presented in \cite{HsSaSa:2016} for elastic-acoustic interactions, \cite{SaSa:2016} for acoustic-piezoelectric interactions, and \cite{HsSaSaWe:2016} for acoustic-thermoelastic interactions, since all those results can be recovered from the ones in this communication by setting to zero selected entries of the piezoelectric tensor, or thermal consants. Moreover, the present work complements the recent articles \cite{HsSab:2020,HsSa:2020} where boundary integral equations of the first kind are studied for the dynamic thermo-elastic equations. 

These results can be used to simulate  
wave-structure interactions numerically by using the nowadays well-known convolution quadrature (CQ) method. Numerical experiments based on QC for the special cases of the wave-structure interactions listed above are available in \cite{BrSaSa:2016, HaQiSaSa:2015, HsSaSa:2016, HsSaSaWe:2016, SaSa:2016}. The numerical treatment for the operators in the  present paper will be reported in a separate communication.


\end{document}